%% file: sticks.tex
\documentclass[12pt,a4paper,oneside]{amsart}

\ifdefined\SMART
\usepackage[paperwidth=12cm,paperheight=24cm,left=5mm,right=5mm,top=10mm,bottom=5mm]{geometry}
\else
\calclayout
\fi

\usepackage{amssymb,latexsym,amsfonts,textcomp,fancyhdr,calc,graphicx}
\usepackage{pdfpages,amsfonts,amsthm,amssymb,latexsym,graphicx,leftidx,fancyhdr}
\usepackage{amsmath,amstext,amsthm,amssymb,amsxtra}
\usepackage[allcolors=blue,allbordercolors=blue,pdfborderstyle={/S/U/W 1}]{hyperref}
\usepackage[ocgcolorlinks]{ocgx2}
\usepackage{dsfont}
\usepackage[framemethod=tikz]{mdframed}
\usepackage{asymptote}
\usepackage{enumerate}

\usepackage{cite}

\usepackage{txfonts,pxfonts,tikz} 
\usepackage{tgschola}
\usepackage[T1]{fontenc}

\usepackage{mathtools}

\usepackage{bbm}

\newtheorem{theorem}{Theorem}[section]
\newtheorem{corollary}{Corollary}[section]
\newtheorem{lemma}{Lemma}[section]

\theoremstyle{definition}
\newtheorem{definition}{Definition}[section]
\newtheorem{remark}{Remark}[section]

\newcommand{\beql}[1]{\begin{equation}\label{#1}}
\newcommand{\eeq}{\end{equation}}
\newcommand{\comment}[1]{}

\newcommand{\Abs}[1]{{\left|{#1}\right|}}

\newcommand{\Norm}[1]{{\left\|{#1}\right\|}}

\newcommand{\Qed}{\hfill\mbox{$\Box$}}

\newcommand{\Set}[1]{{\left\{{#1}\right\}}}

\newcommand{\RR}{{\mathbb R}}

\newcommand{\CC}{{\mathbb C}}
\newcommand{\ZZ}{{\mathbb Z}}

\newcommand{\QQ}{{\mathbb Q}}

\newcommand{\one}{{\mathbbm{1}}}

\newcommand{\inner}[2]{{\langle #1, #2 \rangle}}

\newcommand{\supp}{{\rm supp\,}}

\newcommand{\ft}[1]{\widehat{#1}}

\newcounter{rem}
\newcounter{step}
\newcounter{othm}
\def\theothm{\Alph{othm}}
\newenvironment{othm}{
  \em
  \vskip 0.10in
  \refstepcounter{othm}
  \noindent{\bf Theorem\ \theothm}
}

\newcounter{mysec}
\newcounter{mysubsec}[mysec]
\usepackage{soul}

\begin{document}

\sloppy

\title[Spectrality of two line segments]{Spectrality of a measure consisting of two line segments}

\author{Mihail N. Kolountzakis}
\address{\href{http://math.uoc.gr/en/index.html}{Department of Mathematics and Applied Mathematics}, University of Crete,\\Voutes Campus, 70013 Heraklion, Greece,\newline and \newline \href{https://ics.forth.gr/}{Institute of Computer Science}, Foundation of Research and Technology Hellas, N. Plastira 100, Vassilika Vouton, 700 13, Heraklion, Greece}
\email{kolount@gmail.com}

\author{Sha Wu}
\address{\href{http://math.uoc.gr/en/index.html}{Department of Mathematics and Applied Mathematics}, University of Crete,\\Voutes Campus, 70013 Heraklion, Greece.\newline and \newline
\href{https://math.hnu.edu.cn/index.htm}{School of Mathematics, Hunan University}, Changsha 410082, People’s Republic of China}
\email{shaw0821@163.com}

\makeatletter
\@namedef{subjclassname@2020}{\textup{2020} Mathematics Subject Classification}
\makeatother


\begin{abstract}
Take an interval $[t, t+1]$ on the $x$-axis together with the same interval on the $y$-axis and let $\rho$ be the normalized one-dimensional Lebesgue measure on this set of two segments. Continuing the work done by Lev \cite{lev2018fourier}, Lai, Liu and Prince \cite{lai2021spectral} as well as Ai, Lu and Zhou \cite{ai2023spectrality} we examine the spectrality of this measure for all different values of $t$ (being spectral means that there is an orthonormal basis for $L^2(\rho)$ consisting of exponentials $e^{2\pi i (\lambda_1 x + \lambda_2 y)}$). We almost complete the study showing that for $-\frac12<t<0$ and for all $t \notin \QQ$ the measure $\rho$ is not spectral. The only remaining undecided case is the case $t=-\frac12$ (plus space). We also observe that in all known cases of spectral instances of this measure the spectrum is contained in a line and we give an easy necessary and sufficient condition for such measures to have a line spectrum.
\end{abstract}

\date{January 20, 2025}
\thanks{The author Sha Wu is also supported by Hunan Provincial Innovation Foundation for Postgraduate(LXBZZ2024024).}
\maketitle

\tableofcontents

\section{Introduction}\label{s:intro}
\subsection{Spectrality and the Fuglede Conjecture}\label{ss:spectrality}
Let $\mu$ be a Borel probability measure on $\mathbb{R}^{d}$ with compact support $\Omega$. The measure $\mu$ is called \emph{spectral} if  there exists  a countable set $\Lambda\subset \mathbb{R}^{d}$  such that 
$$
E_\Lambda:=\{e^{2\pi i \lambda\cdot x}:\lambda\in \Lambda\}  \ \text{ forms an orthogonal basis for}  \ L^{2}(\mu),
$$
i.e.,  there exists  a countable set $\Lambda\subset \mathbb{R}^{d}$  such that 
$$
f(x)=\sum_{\lambda \in  \Lambda} c_{\lambda}(f) e_\lambda(x), \text{ with } e_\lambda(x) = e^{2\pi i \lambda\cdot x},
$$
for any $f\in\ L^{2}(\mu)$, where
$$
c_{\lambda}(f) = \inner{f}{e_\lambda}_\mu = \int  f(x)e^{-2 \pi i \lambda\cdot x}\, d\mu(x).
$$
In this case, we  call $\Lambda$ a spectrum of $\mu$.
In particular, if $\mu$ is the Lebesgue measure restricted on the set $\Omega$ of Lebesgue measure 1,  then we say $\Omega$ is a spectral set.  (The definition is trivially extended to finite nonnegative Borel measures, not necessarily probability measures, and sets of finite Lebesgue measure.)

Spectral sets were first introduced by Fuglede\cite{fuglede1974operators} who proposed the {\it Fuglede Conjecture} (also known as {\it Spectral Set Conjecture}).

{\bf{Fuglede Conjecture:}} $\Omega\subset \mathbb{R}^d$ is a  spectral set  if and only if  it tiles $\mathbb{R}^d$  by translation.
The set $\Omega $ is said to tile $\mathbb{R}^{d}$ by translations if there exists a discrete set $L\subset \mathbb{R}^d $ such that  $$\bigcup_{l \in L}(\Omega+l)=\mathbb{R}^d \quad \text { and } \quad m((\Omega+l_1) \cap (\Omega+l_2))=0 \  \text { for all }l_1 \neq l_2\in L,$$ where $m(\cdot)$ denotes the Lebesgue measure and $L$ is called a tiling complement of $\Omega$.

The classical example in dimension 1 is the interval $\Omega=[0,1]$ which is both a spectral set and  a translational tile. More precisely, the set $\Lambda=\mathbb{Z}$ serves simultaneously as a spectrum and a tiling complement of $\Omega$, and $f(x)=\sum_{n \in \mathbb{Z}} c_{n}(f) e^{2 \pi i n x}$ for any $f\in\ L^{2}([0,1])$. For more general examples,  Fuglede \cite{fuglede1974operators,fuglede-ball} proved that the conjecture is true in the case of a triangle or a disk in the plane (both sets are neither spectral nor tiles), and he also proved that $\Omega$ can tile with a lattice tiling complement $L$ if and only if the dual lattice $L^*$ is a spectrum for $\Omega$.
 
It has been proved \cite{tao2004fuglede,kolountzakis2006tiles,kolountzakis2006hadamard,farkas2006onfuglede,farkas2006tiles} that both directions of the conjecture are not valid when the dimension is at least 3, but  the conjecture is still open in both directions  in dimensions 1 and 2. Although the conjecture is not correct in high dimensions, it has been an important topic of research and there are many positive and negative results about the relation of tiling to spectrality. For example, an important result   was  recently proved by Lev and Matolcsi\cite{lev2022fuglede}, who showed that the conjecture holds in any dimension for a convex body. For some cyclic groups, this conjecture is also true if some appropriate conditions are restricted, and more results can be found in \cite{laba2002spectral, zhang2023fuglede,malikiosis2017fuglede,kiss2022fuglede}. See also the recent survey \cite{kolountzakis2024orthogonal} for a more thorough description of the problem and the existing results.
  
The goal of the paper is to study the spectrality of a class of so-called {\it symmetric additive measures}. They are described in the next section.

\subsection{Symmetric additive measures. Previous work.}\label{ss:measures}
Recall that a Borel measure $\mu$  is continuous if $ \mu(\{x\})=0 $ for all $x\in\mathbb{R}$.
 
\begin{definition}\label{de(1-1)}Let $\mu$ be a continuous  Borel measure on $\mathbb{R}$.   The symmetric additive measure for $\mu$ is the
probability measure $\rho$ on $\RR^2$ given by
$$
\rho= \mu\times \delta_0 + \delta_0\times\mu ,
$$
where $\delta_0$ is the Dirac measure at $0$. 
\end{definition}

The study of exponential bases for additive measures was initiated in \cite{lev2018fourier} and continued in \cite{lai2021spectral} and in \cite{ai2023spectrality}.

We are interested in the special case where $\mu$ is  Lebesgue measure where  $\mu$ is one-half of Lebesgue measure supported on the unit interval $[t, t+1]$ (see Fig.\ \ref{fig:rho}). When one wants to know the spectrality or not of this measure it is enough, by the symmetry of the problem, to consider only the cases $t \ge -\frac12$ of the parameter $t$.

The following are some of the results proved in \cite{lai2021spectral, ai2023spectrality} about this measure.

\begin{othm} \label{T(1-1)}
{\rm (\cite{lev2018fourier,lai2021spectral,ai2023spectrality})}
\ Let $\rho=\mu\times \delta_0 + \delta_0\times\mu$ be a  symmetric additive measure, where the measure $\mu$ is one-half of Lebesgue measure supported on $[t, t+1]$.
\begin{enumerate}
\item  If $-\frac{1}{2}<t<0$ and  $2t+1=\frac{1}{a}$, where $a >1$ is a positive integer, then  $\rho$ is not spectral.
\item If  $t\in  \mathbb{Q}  \setminus\{-\frac{1}{2}\}$, then $\rho$ is a spectral measure if and only if $t\in \frac{1}{2}\mathbb{Z}$. In this case, $\rho$ has a unique spectrum  of the form
$$\Lambda=\{(\lambda,-\lambda): \lambda\in\Lambda_0\},$$
where $\Lambda_0$ is the spectrum of the Lebesgue measure supported on $[-t-1, -t]\cup[t, t+1]$.
\item If $E_{\Lambda}$ is a Riesz basis for a symmetric additive space $L^2(\rho)$, then at least one of $E_{\Lambda_{x}}$, $E_{\Lambda_{y}}$ is not a Riesz basis for the component space $L^2(\mu)$, where $\Lambda_{x}$, $\Lambda_{y}$ are the projections of $\Lambda$ onto the $x-$, and $y-$ axis.
\item For $t=-1/2$ (``plus space'') any frame of exponentials cannot have frequencies contained in a straight line.
\item The measure $\rho$ admits a frame of exponentials for all $t$.
\end{enumerate}
\end{othm}

\begin{figure}[h]
\ifdefined\SMART\resizebox{10cm}{!}{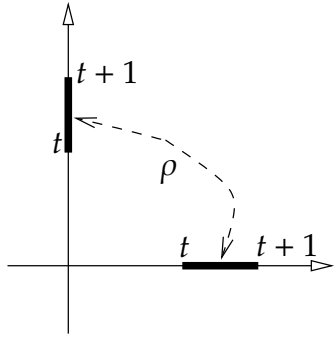}\else
\input rho.pdf_t
\fi
\caption{The symmetric additive measure we consider in this paper. $\rho$ is a probability measure. On each of the two unit-length segments it equals 1/2 Lebesgue measure. The segments may or may not intersect. By symmetry it is enough to consider the cases $t \ge -\frac12$.}\label{fig:rho}
\end{figure}

\subsection{Main results}\label{ss:main-results}

The following answers Question 1 of \cite[\S 7]{lai2021spectral} except for the case $t = -1/2$ (``plus space'', according to \cite{lai2021spectral}), which remains open. Note that the case $t=0$ (the $L$-shape of \cite{lai2021spectral}) does have a spectrum, the set $\Set{(n/2, -n/2),\ \ n \in \ZZ}$, as shown in \cite{lai2021spectral}, and also in our Section \ref{s:projections}.

\begin{theorem}\label{th:intersecting}
If $\rho$ is the probability measure $\mu\times\delta_0+\delta_0\times\mu$, where $\mu$ is one-half Lebesgue measure on the interval $[t, t+1]$ with $t \in (-\frac12, 0)$, then $\rho$ is not spectral.
\end{theorem}

This Theorem proves the non-spectrality of all the cases where the two segments intersect except the case when the two segments intersect in their midpoint. We prove Theorem \ref{th:intersecting} in Section \ref{s:intersecting}.

All the spectra of the measure $\rho$ that are known \cite{lai2021spectral,ai2023spectrality}, for any value of $t$,  belong to the straight line $y = -x$. Our next result describes precisely when it is possible for $\rho$ to have a spectrum contained in a straight line. (Notice we can always assume the straight line goes through the origin as all spectra can be translated to contain the origin.)

Let $L$ be a straight line through the origin and $u$ be a unit vector along $L$. Let us also denote by $u^\perp$ the orthogonal subspace to $L$ (a straight line also). We denote by $\pi_L$ the orthogonal projection operator onto line $L$ (but taking values in $\RR$). In other words $\pi_L(v) = t$ for any $v \in t u + u^\perp$.

If $\nu$ is a Borel measure on $\RR^2$ then the projection of $\nu$ onto $L$ is the measure $\pi_L\nu$ on $\RR$ defined by
$$
\pi_L\nu(E) = \nu(E u + u^\perp),
$$
where $E \subseteq \RR$.


\begin{theorem}\label{th:projections}
Suppose $\rho$ is a probability measure on $\RR^2$ whose support is a finite union of line segments. Suppose also that $L$ is a straight line through the origin such that the orthogonal projection $\pi_L$ onto $L$ is one-to-one $\rho$-almost everywhere.

Then $\rho$ has a spectrum $\Lambda u \subseteq L$ if and only if the projection measure $\pi_L\rho$ has spectrum $\Lambda \subseteq \RR$.
\end{theorem}

We prove Theorem \ref{th:projections} in Section \ref{s:projections}. We also show there that all known cases of spectral measures $\rho$ of the above type are simple consequences of Theorem \ref{th:projections} and known results about one-dimensional spectral sets such as the characterization in \cite{laba2001twointervals} of which sets that are unions of two intervals are spectral (exactly those that tile).

In the negative direction again we complete the study of this problem by proving the following.
\begin{theorem}\label{th:t-irrational}
If $\rho$ is the probability measure $\mu\times\delta_0+\delta_0\times\mu$, where $\mu$ is one-half Lebesgue measure on the interval $[t, t+1]$ with $t \notin \QQ$, then $\rho$ is not spectral.
\end{theorem}
Notice that the only case that is left undecided after our results is the case $t=-1/2$, the plus space. In Section \ref{s:t-irrational} we prove Theorem \ref{th:t-irrational}.

In Section \ref{s:zero-set} we describe some properties of the zero set of the Fourier Transform of the measure $\rho$ which will be useful in the proofs that follow.

{\bf Acknowledgement:} The first author would like to thank Chun-Kit Lai for showing him the problem and related results during a visit in November 2023.

\section{The zero set and the spectrum}\label{s:zero-set}

Let $\rho$ be the measure $\mu\times\delta_0+\delta_0\times\mu$, where $\mu$ is one-half of Lebesgue measure on $[t, t+1]$, for some $t\in\RR$ so that $\rho$ is a probability measure. The set of zeros of $\ft{\rho}$ is easily \cite{lai2021spectral,ai2023spectrality} seen to be the set
\begin{align}
Z(\rho) &= \Set{\lambda: \ft{\rho}(\lambda) = 0} \nonumber \\
 &= \Set{\lambda=(\lambda_1, \lambda_2): e^{\pi i (\lambda_1-\lambda_2)(2t+1)} \frac{\sin\pi\lambda_1}{\pi\lambda_1} = - \frac{\sin\pi\lambda_2}{\pi\lambda_2}}.\label{zero-set}
\end{align}
(Notice that the value of the function $\frac{\sin \pi x}{\pi x}$ at $0$ is $1$.)

Suppose $\rho$ is spectral with spectrum $\Lambda \subseteq \RR^2$, with $0 \in \Lambda$. Assume also that $t \neq -1/2$, so that we exclude from our discussion in this section the case of the plus space.
We conclude that
\begin{equation}\label{two-groups}
\Lambda \subseteq \Lambda-\Lambda \subseteq \Set{0} \cup Z(\rho) \subseteq H_1 \cup H_2
\end{equation}
where the two \textit{subgroups} $H_1, H_2$ of $\RR^2$ are
$$
H_1 = \ZZ^2,\ \ H_2 = \Set{(\lambda_1, \lambda_2):\ \lambda_1-\lambda_2 \in \frac{1}{2t+1}\ZZ}.
$$
Indeed, for $\lambda = (\lambda_1, \lambda_2)$ to be in $Z(\rho)$ the factor $e^{\pi i (\lambda_1-\lambda_2)(2t+1)}$ must be real or both sines must be 0. The second case means that $\lambda \in \ZZ^2$ while the first case implies that $(\lambda_1-\lambda_2)(2t+1)$ must be an integer. The group $H_2$ consists of equispaced parallel lines perpendicular to the line $y=-x$.

It follows from Lemma 11.4 in \cite{greenfeld2017fuglede} that
$$
\Lambda \subseteq H_1 \text{ or } \Lambda \subseteq H_2.
$$
From Theorem 4.2 in \cite{lai2021spectral} (see also our Remark \ref{rem:z2} below) it follows that $\Lambda \subseteq H_1 = \ZZ^2$ is not true, so we conclude that $\Lambda \subseteq H_2$.

We have proved:
\begin{theorem}\label{th:lines}
If $\rho$ is the probability measure $\mu\times\delta_0+\delta_0\times\mu$, where $\mu$ is one-half Lebesgue measure on the interval $[t, t+1]$, with $t \neq -1/2$, and $\rho$ is spectral with spectrum $\Lambda \subseteq \RR^2$ and $0 \in \Lambda$ then for every $\lambda = (\lambda_1, \lambda_2) \in \Lambda$ there exists an integer $k(\lambda)$ such that
\begin{equation}\label{lambda-diff}
\lambda_2-\lambda_1 = \frac{k(\lambda)}{2t+1}.
\end{equation}
\end{theorem}

\begin{remark}\label{rem:z2}
The fact that the multiplicity of $\Lambda$ is 1 (all points of $\Lambda$ project uniquely onto the coordinate axes) is the easy part of Theorem 4.2 in \cite{lai2021spectral}, while the fact that $\Lambda$ cannot be a subset of $\ZZ^2$ is more involved in \cite{lai2021spectral} and uses the proof of Theorem 1.2 therein. Let us show here an easy proof of the latter result using the fact that the multiplicity is 1.

Take the function $f \in L^2(\rho)$ which is 1 on the horizontal segment and 0 on the vertical segment. It does not matter if the two segments of $\rho$ intersect as they will intersect on at most one point which has $\rho$-measure 0. If $0 \in \Lambda \subseteq \ZZ^2$ is a spectrum of $\rho$ then $f = \sum_{\lambda \in \Lambda} \inner{f}{e_\lambda} e_\lambda$. But
\begin{align*}
\inner{f}{e_\lambda} &= \int f(x, y) e^{-2\pi i(\lambda_1 x + \lambda_2 y)} \,d\rho(x, y)\\
&= \frac12\int_t^{t+1} e^{-2\pi i \lambda_1 x}\,dx\\
&= \begin{cases} 1/2 & \text{ if } \lambda_1 = 0, \\ 0 & \text{ otherwise.} \end{cases}
\end{align*}
The multiplicity of $\Lambda$ being 1 means that there is at most one point of $\Lambda$ with $\lambda_1=0$, therefore this point is the origin. From the expansion of $f$ with respect to $\Lambda$ it follows that $f$ is constant (the series has one term only), a contradiction.
\end{remark}

\section{Intersecting line segments}\label{s:intersecting}

In this section we prove Theorem \ref{th:intersecting}. Under the assumptions of that theorem let us assume that $\Lambda$ is a spectrum of $\rho$ containing $0$. We will arrive at a contradiction.

From Theorem \ref{th:lines} it follows that for any $\lambda=(\lambda_1, \lambda_2) \in \Lambda$ we have
\begin{equation}\label{diff}
\lambda_2-\lambda_1 \in \frac{1}{2t+1}\ZZ.
\end{equation}
Let now $f(x) = \sum_{\lambda\in\Lambda} c_\lambda(f) e_\lambda(x)$, where $c_\lambda(f)\in \ell^2(\Lambda)$.  It follows from \eqref{diff} that
\begin{equation}\label{period}
f(x+T) = f(x),\ \ \text{ as functions in } L^2(d\rho),
\end{equation}
where $T = (2t+1, -2t-1)$. Since $\Lambda$ is assumed to be a spectrum of $\rho$ it follows that every function $f \in L^2(d\rho)$ satisfies the periodicity condition \eqref{period}.

\begin{figure}[h]
\ifdefined\SMART\resizebox{10cm}{!}{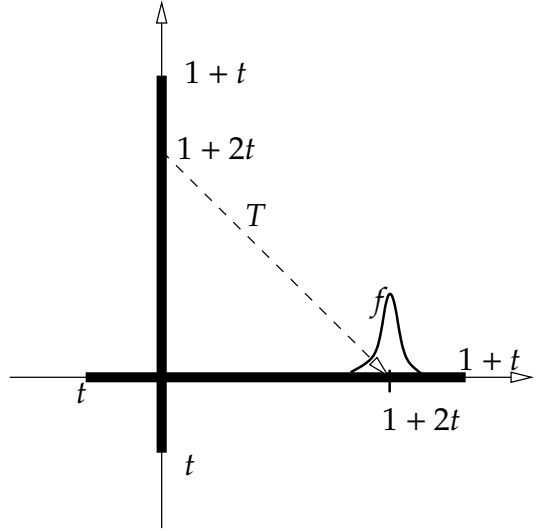}\else
\input crossing.pdf_t
\fi
\caption{The measure $\rho$ in the case $-1/2<t<0$, with the smooth function $f$ used in the proof.}\label{fig:crossing}
\end{figure}

Since $-\frac12 < t < 0$ we have $0 < 1+2t < 1+t < 1$ and the two points $(0, 1+2t)$ and $(1+2t, 0)$ both belong to the interior of the line segments comprising $\supp\rho$. The difference of these two points is $T$, the period vector appearing in \eqref{period}. However \eqref{period} is true for almost all $x$ so we need some more work to arrive to a contradiction.

Define the function $f \in L^2(d\rho)$ to be 0 on the vertical segment, equal to 1 at the point $(1+2t, 0) \neq (0, 0)$ and, restricted to the open horizontal segment, to be a smooth function. In other words take $f$ to be a smooth bump function supported close to $(1+2t, 0)$. (See Fig.\ \ref{fig:crossing}.) We have the $L^2(d\rho)$ expansion
\begin{equation}\label{expansion}
f(x) = \sum_{\lambda\in\Lambda} \inner{f}{e_\lambda} e_\lambda(x).
\end{equation}
But $\inner{f}{e_\lambda} = \frac12 \int_t^{t+1} f(s) e^{-2\pi i \lambda_1 s}\,ds$ and, since $f$ is smooth, we have
$$
\Abs{\inner{f}{e_\lambda}} = O(\Abs{\lambda_1}^{-10}).
$$
From Theorem 4.2 in \cite{lai2021spectral} we also have that there is exactly one $\lambda \in \Lambda$ for each $\lambda_1$ appearing (in the terminology of that paper, $\Lambda$ has multiplicity 1). From Theorem 4.1 in \cite{lai2021spectral} we have that the set of $\lambda_1$ is a frame for Lebesgue measure on $[t, t+1]$, so it must have bounded density (this follows also from the tiling property proved in Lemma \ref{l:tiling} below). These two facts imply that $\sum_{\lambda\in\Lambda} \Abs{\inner{f}{e_\lambda}} < \infty$ and from this we obtain that \eqref{expansion} holds for all $x\in\supp\rho$ (in fact, for all $x \in \RR^2$) as both sides are continuous functions. From \eqref{period} we must then have
$$
1 = f(2t+1, 0) = f(0, 2t+1) = 0,
$$
a contradiction. 

\section{Line spectra from projections}\label{s:projections}

Here we prove Theorem \ref{th:projections} and see how it can be applied to produce line spectra for some collections of measures supported on line segments.

Let $u$ be a unit vector in the straight line $L$ that goes through the origin. By our assumption on the injectivity of $\pi_L$ $\rho$-almost everywhere, any function $f(x)$ on $\supp \rho$ can be written as
$$
f(x) = \widetilde{f}(u\cdot x),
$$
for $\rho$-almost all $x$, where $\widetilde{f}:\RR\to\CC$ is supported on $u\cdot\supp\rho$. Also by the injectivity of $\pi_L$ we have that $\int_{\RR^2}\Abs{f}^2d\rho = \int_\RR\Abs{\widetilde{f}}^2d\pi_L\rho$. Hence the map $f \to \widetilde f$ is a Hilbert space isometry $L^2(\rho) \to L^2(\pi_L \rho)$ and inner products are also preserved.

Next, observe that if $\lambda u \in L$ for some $\lambda\in\RR$, then  we have that $e_{\lambda(x) u} = e^{2\pi i \lambda u \cdot x}$, $x \in \RR^2$, is constant on lines perpendicular to $L$. In other words, $e_{\lambda u} (x) = \widetilde{e_\lambda}(\pi_L(x)) = e_\lambda(\pi_L(x))$.

This implies that $\Lambda \subseteq \RR$ is a spectrum for $\pi_L\rho$ if and only if $\Lambda u \subseteq \RR^2$ is a spectrum for $\rho$, as we had to show.

\Qed

The following result  if valuable in determining when the projection measure is spectral.
\begin{othm}\label{th:abs-cont}
(\cite{dutkay2014uniformity}[Corollary 1.4])
If a measure $\mu$ on $\RR^d$, absolutely continuous with respect to Lebesgue measure, is spectral (or even has a tight frame of exponentials) then it is a constant multiple of Lebesgue measure. 
\end{othm}

Using Theorem \ref{th:abs-cont} in conjunction with Theorem \ref{th:projections} allows us to easily determine the existence of line spectra.

\begin{corollary}\label{c:two-segments}
If $\rho$ is a probability measure in the plane which consists of one-half the Lebesgue measure on the line segment from $(t, 0)$ to $(t+1, 0)$ and one-half the Lebesgue measure on the line segment from $(0, t)$ to $(0, t+1)$, then if $0 \le t\in\frac12\ZZ$ the measure $\rho$ is spectral and with a spectrum contained in the line $y = -x$.
\end{corollary}

\begin{proof}
Projecting $\rho$ onto the line $L$ given by $y = -x$ we see that the projection measure is supported on the union of two intervals
$$
U = \frac1{\sqrt2} \left( (-(t+1), -t) \cup (t, t+1) \right),
$$
and is constant on $U$. From Theorem \ref{th:projections} and Theorem \ref{th:abs-cont} it is enough to show that $U \subseteq \RR$ is spectral. See Fig.\ \ref{fig:two-sticks}.

\begin{figure}[h]
\ifdefined\SMART\resizebox{10cm}{!}{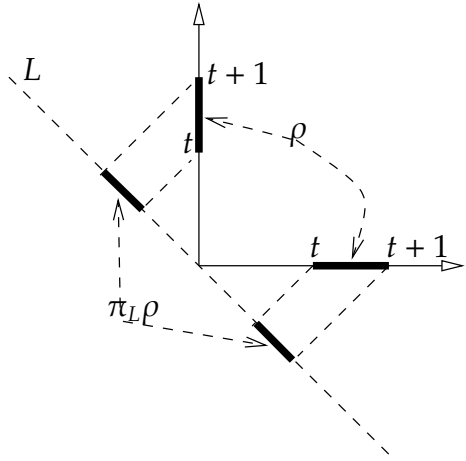}\else
\input two-sticks.pdf_t
\fi
\caption{We project the two line segments onto $L$. If the resulting two intervals tile the line, then they are spectral and so is the measure $\rho$.}\label{fig:two-sticks}
\end{figure}

By the result of \cite{laba2001twointervals} which verifies the Fuglede Conjecture when the set consists of two intervals, $U$ is spectral if and only if it tiles the real line. If $t=0$ the union is one interval only, which certainly tiles the line. If $t>0$ and $2t \in \ZZ$ then the gap between the two intervals is of length $2t/\sqrt2$, which is an integer multiple of the length of each of the two intervals which have length $1/\sqrt2$. In this case one can easily see that $U$ tiles the line by first showing the it can tile an interval. To tile an interval we just take $k+1=2t+1$ copies of $U$ translated at the points $0, \frac1{\sqrt2}, \ldots, \frac{k}{\sqrt2}$.

\end{proof}

\begin{remark}\label{rem:easy}
Notice that our approach provides an easy way to complete the proof in \cite{ai2023spectrality} after the essential Proposition 8 in that paper has been proved which says that any spectrum of $\rho$ in the case $t \in \QQ\setminus\Set{-1/2}$ must be contained in a straight line. In Corollary \ref{c:two-segments} we saw that a line spectrum exists if $0<t \in\frac12 \ZZ$. And if $t \in \QQ$ any spectrum must be contained in a straight line \cite[Prop.\ 8]{ai2023spectrality}. But the only line onto which $\rho$ projects {\em injectively} to a function that is constant on its support is the line $y=-x$. And the projection on that line is a union of two equal intervals which can tile the line (equivalently, is spectral in the line) only when $t \in  \frac12\ZZ$.
\end{remark}

The method of projections is quite flexible when one seeks to determine if $\rho$ is spectral with a line spectrum. Take for instance two arbitrary non-intersecting line segments in the plane, equipped with a constant multiple of Lebesgue measure each, as shown in Fig.\ \ref{fig:loaded}.

\begin{figure}[h]
\ifdefined\SMART\resizebox{10cm}{!}{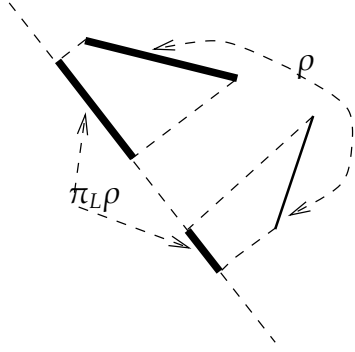}\else
\input loaded.pdf_t
\fi
\caption{We choose a line $L$ onto which to project $\rho$ so that the projection meeasure is constant on its support.}\label{fig:loaded}
\end{figure}
 
So $\rho$ is the sum of Lebesgue measure on one interval multiplied by some constant $c_1>0$ and Lebesgue measure on the other interval multiplied by some constant $c_2 > 0$. We then choose a line $L$ such that the projection measure $\pi_L\rho$ is constant on its support (which is one or two intervals). If such a line cannot be found (this depends on the slopes of the segments and the constants $c_1, c_2$) then $\rho$ does not have a line spectrum. If such a line is found then we examine if the support of the projection can tile the line. If it does then its spectrum is also a spectrum of $\rho$. If not then $\rho$ does not have a spectrum on this line.

And the method need not be restricted to two segments. In the cases shown in Fig.\ \ref{fig:cases} it is easy to find a line onto which $\rho$ projects to a measure constant on its support and this support is itself spectral in the line, thus implying that $\rho$ has a line spectrum. In all three examples shown we project onto the $x$-axis. The projection function is constant on its support. This support is a single interval in the first two cases from the left. For the case on the right the support consists of two intervals which together tile periodically with period $\ell$, and so we have a spectrum in this case as well.

\begin{figure}[h]
\ifdefined\SMART\resizebox{10cm}{!}{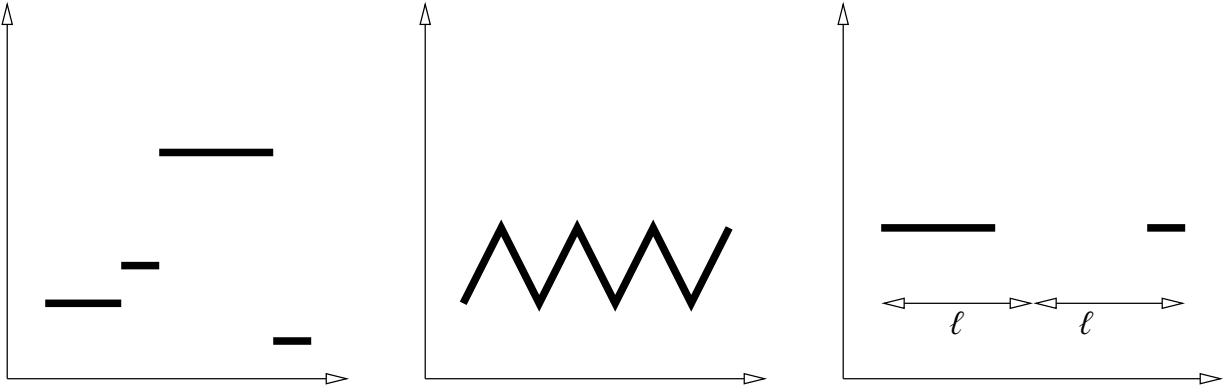}\else
\input cases.pdf_t
\fi
\caption{These measures all have a spectrum contained in the $x$-axis.}\label{fig:cases}
\end{figure}

\section{The case of irrational $t$}\label{s:t-irrational}

In this section we prove Theorem \ref{th:t-irrational}. In what follows $\rho$ is the probability measure $\mu\times\delta_0+\delta_0\times\mu$ where $\mu$ is one-half Lebesgue measure on the interval $[t, t+1]$, where $t \in \RR\setminus\Set{-\frac12}$ is fixed.

Our first goal is to show that for all points $\lambda=(\lambda_1, \lambda_2) \in \Lambda$ the two coordinates are comparable.

\begin{lemma}\label{l:bounds}
If $\Lambda$, with $0 \in \Lambda$, is a spectrum of $\rho$ then it is not possible to have an infinite sequence $\lambda^n=(\lambda^n_1, \lambda^n_2) \in \Lambda$ tending to infinity such that $\lambda^n_2 = o(\lambda^n_1)$ or $\lambda^n_1 = o(\lambda^n_2)$.

In other words there is a constant $K>1$ such that for all $\lambda \in \Lambda\setminus\Set{(0, 0)}$ we have
\begin{equation}\label{K-bound}
K^{-1} \Abs{\lambda_1} \le \Abs{\lambda_2} \le K \Abs{\lambda_1}.
\end{equation}
\end{lemma}

\begin{proof}
For any two different points $\lambda=(\lambda_1, \lambda_2), \nu=(x, y) \in \Lambda$ we have, from \eqref{zero-set}, that
\begin{equation}\label{sizes}
\Abs{\frac{\sin\pi(x-\lambda_1)}{\pi(x-\lambda_1)}} = \Abs{\frac{\sin\pi(y-\lambda_2)}{\pi(y-\lambda_2)}}.
\end{equation}
The set $\Lambda$ is infinite and there is a positive lower bound on the distance of any two of its points (this is true for all spectra), so it must be the case that its points tend to infinity.

Fix a  point $(\lambda_1, \lambda_2) \in \Lambda\setminus\Set{(0, 0)}$ and apply \eqref{sizes} while $\nu = (x, y) \to \infty$, with $y = o(x)$, is a point of $\Lambda$.
It is clear that the left hand side is $o(\cdot)$ of the right hand side, unless $\sin\pi(y-\lambda_2) \to 0$ (for some subsequence of the points $\nu$, which we may consider to be the whole sequence), which is equivalent to $\Set{y-\lambda_2}\to 0$, where $\Set{\cdot}$ denotes the fractional part. Using the same reasoning with the point $(0, 0) \in \Lambda$ in place of $(\lambda_1, \lambda_2)$ we also obtain that $\Set{y} \to 0$. Together these two imply that $\lambda_2 \in \ZZ$. Thus we showed that all $\lambda_2 \in \ZZ$. From \eqref{sizes} with $(x, y) = (0, 0)$ it follows that if $\lambda_2 = 0$ then we also have $\lambda_1 = 0$ and if $\lambda_2 \in \ZZ\setminus\Set{0}$ then we also have $\lambda_1 \in \ZZ\setminus\Set{0}$. We have proved that $\Lambda\subseteq\ZZ^2$ which is impossible by the results in \cite[Theorem 4.2]{lai2021spectral} (but see also our Remark \ref{rem:z2}). This concludes the proof that we cannot find a sequence $\nu = (x, y) \in \Lambda$ tending to infinity with $y = o(x)$ or, by symmetry, $x = o(y)$. Also by \cite[Theorem 4.2]{lai2021spectral} it follows that for $\lambda \in \Lambda\setminus\Set{(0, 0)}$ we have $\lambda_1 \neq 0$ and $\lambda_2 \neq 0$ (multiplicity one, in the language of \cite{lai2021spectral}). From this observation and the impossibility of $\lambda_2 = o(\lambda_1)$ or $\lambda_1 = o(\lambda_2)$ the bound \eqref{K-bound} follows for some $K>1$.

\end{proof}

Next, we show that the spectrality assumption for $\rho$ leads to a certain one-dimensional tiling, which will allow us to deduce properties of the projection of $\Lambda$ to the coordinate axes.

\begin{lemma}\label{l:tiling}
If $\Lambda$ is a spectrum of $\rho$ then we have the level-2 tiling of the real line
$$
2 = \sum_{\lambda \in \Lambda} \Abs{\ft{\one_{[-\frac12, \frac12]}}}^2(x-\lambda_1)\ \ \ (x \in \RR).
$$
\end{lemma}

\begin{proof}
Consider the function $f \in L^2(\rho)$ which is $0$ on the $y$-axis and equal to $\phi$ on $[t, t+1]\times\Set{0}$ for some $\phi \in L^2([t, t+1])$. Expanding $f$ with respect to $\Lambda$ we get
$$
2\Norm{\phi}_{L^2([t, t+1])}^2 = \sum_{\lambda \in \Lambda} \Abs{\ft{\phi}(\lambda_1)}^2.
$$
Picking $\phi(s) = \one_{[t, t+1]}(s)e^{2\pi i x s}$ for some $x \in \RR$ we get
\begin{equation}\label{2-tiling}
2 = \sum_{\lambda\in\Lambda} \Abs{\ft{\one_{[-\frac12, \frac12]}}}^2(x-\lambda_1).
\end{equation}

\end{proof}

The next Lemma shows the very crucial property of the set $\Lambda_1$, that it is of finite complexity: in any fixed-length window on the line we may only see a finite set of different patterns of $\Lambda_1$.

\begin{lemma}\label{l:finite-complexity}
The tiling in Lemma \ref{l:tiling} is of finite complexity. This means that there are finitely many different gaps among successive points in $\Lambda_1 = \Set{\lambda_1: \lambda=(\lambda_1, \lambda_2) \in \Lambda}$.
\end{lemma}

\begin{proof}
Let us write the set $\Lambda= \Set{\lambda^n: n \in \ZZ}$ in increasing order of the first coordinates as follows
$$
\cdots \le \lambda^{-1}_1 \le \lambda^0_1 = 0 \le \lambda^1_1 \le \lambda^2_1 < \cdots.
$$
Our first goal is to show that $\lambda^1_1$ can take only finitely many values. 

By the tiling property, there is an absolute constant $C$ such that $\lambda^1_1 \le C$, so $\lambda^1$ belongs to the vertical strip
$$
V = \Set{(x, y): 0\le x\le C}.
$$
From \eqref{K-bound} we also know that $\lambda^1$ belongs to the union of two sectors
$$
S = \Set{(x, y): x\ge 0,\ K^{-1} x \le \Abs{y} \le K x}.
$$
See Fig.\ \ref{fig:regions}.

\begin{figure}[h]
\ifdefined\SMART\resizebox{10cm}{!}{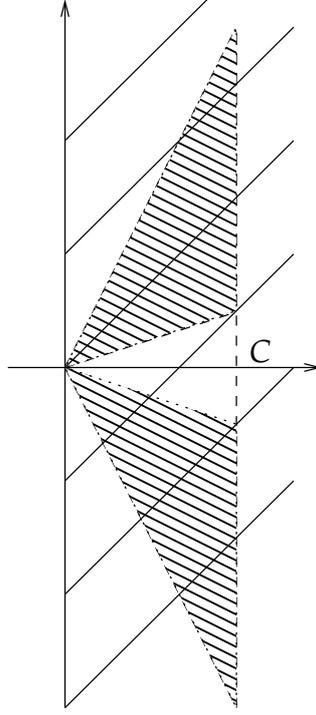}\else
\input regions.pdf_t
\fi
\caption{The shaded region is $V \cap S$. Only finitely many of the lines $y-x=k/(2t+1)$, $k \in \ZZ$, intersect this region and on each of these lines only an interval is contained in the region.}\label{fig:regions}
\end{figure}

Finally we already know that we must have
$$
\lambda^1_2 - \lambda^1_1 = \frac{k}{2t+1}
$$
for some $k \in \ZZ$ (remember that we have assumed $t \neq -1/2$), and it follows that $\Abs{k}$ must be bounded because only finitely many of the straight lines $y-x=k/(2t+1)$ intersect $V \cap S$. For each of the finitely many eligible values of $k$ there are finitely many points on the line $y-x=k/(2t+1)$ which are zeros of $\ft{\rho}$ and belong to $V$, since on each such line the zeros of $\ft{\rho}$ are a discrete set as they correspond the values of $x$ such that
$$
\frac{\sin \pi x}{\pi x} = \pm \frac{\sin \pi (x+\frac{k}{2t+1})}{\pi (x+\frac{k}{2t+1})},
$$
and these are zeros of two entire functions of $x$ that are not identically zero.
It follows that the possible locations for $\lambda^1_1$ are finitely many. Let us call the set of these locations $L \subseteq \RR$.

By translating $\Lambda$ by $-\lambda^n$ we obtain that the only possible values for $\lambda^{n+1}_1-\lambda^n_1$ are again the set $L$. We have proved that the tiling \eqref{2-tiling} is a tiling of finite complexity.

\end{proof}

In the next Lemma we use a result \cite{iosevich2013periodicity} which says that tilings of finite complexity, whose set of translates has a spectral gap, are necessarily periodic. This was first used \cite{iosevich2013periodicity} to prove that spectra of bounded subsets of $\RR$ are necessarily periodic.

\begin{lemma}\label{l:periodic}
The projections of $\Lambda$ onto the $x$- and $y$-axes, $\Lambda_1$ and $\Lambda_2$, are periodic sets, with periods in $\frac12\ZZ$. In other words, there are positive $T_1, T_2 \in \frac12 \ZZ$ such that $\Lambda_1 = \Lambda_1+T_1$ and $\Lambda_2 = \Lambda_2 + T_2$.
\end{lemma}

\begin{proof}
From Lemma \ref{l:finite-complexity} we know that the tiling \eqref{2-tiling} is a tiling of finite complexity. It was proved in \cite{iosevich2013periodicity}, though not stated precisely in this form (see \cite[Theorem 5.1]{kolountzakis2016non}), that in any tiling of finite complexity with a spectral gap, the set of translates is periodic. To have a spectral gap, the distribution $\ft{\delta_{\Lambda_1}}$ needs to vanish on an interval. To see that this is the case here we need to apply \cite{kolountzakis2016non}[Theorem 4.1] with
$$
f = \Abs{\ft{\one_{[-\frac12, \frac12]}}}^2 \in L^1(\RR).
$$
By the conclusion of that result we obtain that
$$
\supp{\ft{\delta_{\Lambda_1}}} \subseteq \Set{0} \cup \Set{\ft{f} = 0}.
$$
But $\ft{f} = \one_{[-\frac12, \frac12]}*\one_{[-\frac12, \frac12]}$ so that $\Set{\ft{f} = 0}$ is the complement of $(-1, 1)$, and we see that $\ft{\delta_{\Lambda_1}}$ vanishes in any proper subinterval of $(0, 1)$. We conclude that $\Lambda_1$ is periodic. Let $T_1>0$ be one of its periods. By symmetry there must also exist a positive real number $T_2$ such that $\Lambda_2 = \Lambda_2 + T_2$.

It is easy to see from the tilings that $T_1, T_2 \in \frac12\ZZ$. Indeed, since the integral of $\Abs{\ft{\one_{[-\frac12, \frac12]}}}^2$ is 1 it follows that the integral of the right hand side of \eqref{2-tiling} over a period, which is $2 T_1$ by looking at the left hand side of \eqref{2-tiling}, must be equal to the total number of copies in a period, call it $n_1$, which gives $2T_1 = n_1$, i.e. $T_1 \in \frac12 \ZZ$, and similarly $T_2 \in \frac12 \ZZ$.

\end{proof}

In the next Lemma we exploit the fact that exponentials with frequencies on the line $y=x$ cannot span $L^2(\rho)$ as the projections of the two segments on that line overlap.

\begin{lemma}\label{l:non-diagonal}
If $\Lambda$ is a spectrum for $\rho$ then $\Lambda$ has infinitely many points not on the line $y=x$.
\end{lemma}

\begin{proof}
The exponentials of the form $e_{(\lambda, \lambda)}(x, y) = e^{2\pi i (\lambda x + \lambda y)}$ are constant along any straight line that is perpendicular to the line $y=x$. Thus the exponentials with frequencies in $\Lambda \cap \Set{(x, y): y=x}$ can only generate part of the subspace of $L^2(\rho)$ which consists of functions that are symmetric with respect to the mapping $(x, y) \to (y, x)$. Clearly this subspace has infinite co-dimension in $L^2(\rho)$ (its orthogonal complement contains all functions which are negated under the transformation $(x, y) \to (y, x)$), so infinitely many exponentials are required outside $y=x$ if the exponentials with frequencies in $\Lambda$ are to generate $L^2(\rho)$.
\end{proof}

The irrationality of $t$ is exploited next to show that there is at most one point of the spectrum on each line $y-x=k/(2t+1)$ for $k \neq 0$.

\begin{lemma}\label{l:only-one-per-line}
If $0 \in \Lambda$ is a spectrum of $\rho$ and $t \notin \QQ$ then on each line of the form
$$
y-x = \frac{k}{2t+1},\ \ \ \text{ for some } k \in \ZZ,
$$
there is at most one point of $\Lambda$.
\end{lemma}

\begin{proof}
It is easy to see from \eqref{zero-set} that
\begin{equation}\label{0-line}
\Set{\ft{\rho}=0} \cap \Set{(x, x): x \in \RR} = \Set{(n, n): 0\neq n\in\ZZ}.
\end{equation}
Since $\Lambda-\Lambda \subseteq \Set{\ft{\rho}=0} \cup \Set{0}$ we deduce that any two different points of $\Lambda$ on the same line
$$
L_k = \Set{(x, y): x-y = \frac{k}{2t+1}},\ \ \ k\in\ZZ,
$$
must differ by a vector of the form $(n, n)$, with $0 \neq n \in \ZZ$.

Let us first assume $k \neq 0$.

We now show that if $t$ is irrational then there is at most one point of $\Lambda$ on each line $L_k$ with $k \neq 0$. Indeed, suppose that
$$
(\lambda_1, \lambda_2), (\lambda_1+\nu, \lambda_2+\nu) \in L_k \cap \Lambda, \text{ for some } \nu \in \ZZ.
$$
Then we have
\begin{align}
e^{\pi i (\lambda_1-\lambda_2)(2t+1)} \frac{\sin\pi\lambda_1}{\pi\lambda_1} &= - \frac{\sin\pi\lambda_2}{\pi\lambda_2}\label{ortho} \\
e^{\pi i (\lambda_1-\lambda_2)(2t+1)} \frac{\sin(\pi\lambda_1+\pi\nu)}{\pi\lambda_1+\pi\nu} &= - \frac{\sin(\pi\lambda_2+\pi\nu)}{\pi\lambda_2+\pi\nu}. \label{ortho1}
\end{align}
Since $\sin(x+\pi\nu) = (-1)^\nu\sin x$ for $x \in \RR$ we can, assuming that both $\lambda_1$ and $\lambda_2$ are not integers, divide \eqref{ortho1} by \eqref{ortho} and get
$$
1+\frac{\nu}{\lambda_1} = 1+\frac{\nu}{\lambda_2}.
$$
Since $k\neq 0$ it follows that $\lambda_1 \neq \lambda_2$ so the last equation implies $\nu=0$, as we had to show. It remains to ensure that $\lambda_1, \lambda_2 \notin \ZZ$. If one of them is an integer then so is the other, by \eqref{ortho}. But if they are both integers then by $\lambda_1-\lambda_2 = \frac{k}{2t+1}$ we obtain that $t \in \QQ$, contrary to our assumption.

Let us now show the result for $k=0$.

Assume that $\nu \in \ZZ\setminus\Set{0}$ and that $(-\nu, -\nu) \in \Lambda$. From Lemma \ref{l:non-diagonal} there exists at least one point of $\Lambda$ outside the line $y=x$. Let us call this point $(\lambda_1, \lambda_2)$. We immediately have equations \eqref{ortho} and \eqref{ortho1} and by the same argument as above we arrive at a contradiction.

\end{proof}

\begin{proof}[Proof of Theorem \ref{th:t-irrational}]
Assuming $\Lambda$ to be a spectrum of $\rho$, with $0 \in \Lambda$, we obtain from Lemma \ref{l:periodic} that
\begin{equation}\label{periodic-sets}
\Lambda_1 = T_1 \ZZ + A,\ \ \Lambda_2 = T_2 \ZZ + B,
\end{equation}
for some positive $T_1, T_2 \in \frac12\ZZ$ and some finite sets $A \subseteq [0, T_1)$ and $B \subseteq [0, T_2)$.

If $\lambda = (\lambda_1, \lambda_2) \in \Lambda$ then, by Theorem \ref{th:lines}, there is $k \in \ZZ$ such that
$$
\lambda_2 - \lambda_1 = \frac{k}{2t+1}.
$$
From \eqref{periodic-sets} we have that $\lambda_1 = mT_1+a$ and $\lambda_2 = n T_2 +b$, for some integers $m, n$ and some $a \in A$, $b \in B$ so, taking fractional parts,
$$
\Set{2\lambda_2-2\lambda_1} = \Set{2b-2a} = \Set{\frac{2k}{2t+1}}.
$$
The quantity $\Set{2b-2a}$ takes only finitely many values as $A, B$ are finite sets. For different $k$ the values of $\Set{\frac{2k}{2t+1}}$ are different since $2/(2t+1)$ is irrational. This means that there are at most finitely many $k$'s for all points of $\Lambda$ which contradicts Lemma \ref{l:only-one-per-line}.

\end{proof}

\bibliographystyle{alpha}
\bibliography{mk-bibliography.bib}

\end{document}

%% file: rho.pdf_t
\begin{picture}(0,0)%
\includegraphics{rho.pdf}%
\end{picture}%
\setlength{\unitlength}{4144sp}%
\begingroup\makeatletter\ifx\SetFigFont\undefined%
\gdef\SetFigFont#1#2#3#4#5{%
  \reset@font\fontsize{#1}{#2pt}%
  \fontfamily{#3}\fontseries{#4}\fontshape{#5}%
  \selectfont}%
\fi\endgroup%
\begin{picture}(1959,2004)(2104,-4978)
\put(3132,-4511){\makebox(0,0)[lb]{\smash{{\SetFigFont{12}{14.4}{\rmdefault}{\mddefault}{\updefault}{\color[rgb]{0,0,0}$t$}%
}}}}
\put(3587,-4506){\makebox(0,0)[lb]{\smash{{\SetFigFont{12}{14.4}{\rmdefault}{\mddefault}{\updefault}{\color[rgb]{0,0,0}$t+1$}%
}}}}
\put(2380,-3889){\makebox(0,0)[lb]{\smash{{\SetFigFont{12}{14.4}{\rmdefault}{\mddefault}{\updefault}{\color[rgb]{0,0,0}$t$}%
}}}}
\put(2522,-3474){\makebox(0,0)[lb]{\smash{{\SetFigFont{12}{14.4}{\rmdefault}{\mddefault}{\updefault}{\color[rgb]{0,0,0}$t+1$}%
}}}}
\put(3016,-4021){\makebox(0,0)[lb]{\smash{{\SetFigFont{12}{14.4}{\rmdefault}{\mddefault}{\updefault}{\color[rgb]{0,0,0}$\rho$}%
}}}}
\end{picture}%

%% file: crossing.pdf_t
\begin{picture}(0,0)%
\includegraphics{crossing.pdf}%
\end{picture}%
\setlength{\unitlength}{4144sp}%
\begingroup\makeatletter\ifx\SetFigFont\undefined%
\gdef\SetFigFont#1#2#3#4#5{%
  \reset@font\fontsize{#1}{#2pt}%
  \fontfamily{#3}\fontseries{#4}\fontshape{#5}%
  \selectfont}%
\fi\endgroup%
\begin{picture}(3129,3174)(664,-3898)
\put(1711,-3571){\makebox(0,0)[lb]{\smash{{\SetFigFont{12}{14.4}{\rmdefault}{\mddefault}{\updefault}{\color[rgb]{0,0,0}$t$}%
}}}}
\put(3331,-2941){\makebox(0,0)[lb]{\smash{{\SetFigFont{12}{14.4}{\rmdefault}{\mddefault}{\updefault}{\color[rgb]{0,0,0}$1+t$}%
}}}}
\put(1711,-1231){\makebox(0,0)[lb]{\smash{{\SetFigFont{12}{14.4}{\rmdefault}{\mddefault}{\updefault}{\color[rgb]{0,0,0}$1+t$}%
}}}}
\put(2881,-3301){\makebox(0,0)[lb]{\smash{{\SetFigFont{12}{14.4}{\rmdefault}{\mddefault}{\updefault}{\color[rgb]{0,0,0}$1+2t$}%
}}}}
\put(2807,-2582){\makebox(0,0)[lb]{\smash{{\SetFigFont{12}{14.4}{\rmdefault}{\mddefault}{\updefault}{\color[rgb]{0,0,0}$f$}%
}}}}
\put(1070,-3138){\makebox(0,0)[lb]{\smash{{\SetFigFont{12}{14.4}{\rmdefault}{\mddefault}{\updefault}{\color[rgb]{0,0,0}$t$}%
}}}}
\put(1666,-1681){\makebox(0,0)[lb]{\smash{{\SetFigFont{12}{14.4}{\rmdefault}{\mddefault}{\updefault}{\color[rgb]{0,0,0}$1+2t$}%
}}}}
\put(2071,-2086){\makebox(0,0)[lb]{\smash{{\SetFigFont{12}{14.4}{\rmdefault}{\mddefault}{\updefault}{\color[rgb]{0,0,0}$T$}%
}}}}
\end{picture}%

%% file: two-sticks.pdf_t
\begin{picture}(0,0)%
\includegraphics{two-sticks.pdf}%
\end{picture}%
\setlength{\unitlength}{4144sp}%
\begingroup\makeatletter\ifx\SetFigFont\undefined%
\gdef\SetFigFont#1#2#3#4#5{%
  \reset@font\fontsize{#1}{#2pt}%
  \fontfamily{#3}\fontseries{#4}\fontshape{#5}%
  \selectfont}%
\fi\endgroup%
\begin{picture}(2724,2724)(1339,-5698)
\put(3010,-3799){\makebox(0,0)[lb]{\smash{{\SetFigFont{12}{14.4}{\rmdefault}{\mddefault}{\updefault}{\color[rgb]{0,0,0}$\rho$}%
}}}}
\put(1437,-3454){\makebox(0,0)[lb]{\smash{{\SetFigFont{12}{14.4}{\rmdefault}{\mddefault}{\updefault}{\color[rgb]{0,0,0}$L$}%
}}}}
\put(3132,-4511){\makebox(0,0)[lb]{\smash{{\SetFigFont{12}{14.4}{\rmdefault}{\mddefault}{\updefault}{\color[rgb]{0,0,0}$t$}%
}}}}
\put(3587,-4506){\makebox(0,0)[lb]{\smash{{\SetFigFont{12}{14.4}{\rmdefault}{\mddefault}{\updefault}{\color[rgb]{0,0,0}$t+1$}%
}}}}
\put(2380,-3889){\makebox(0,0)[lb]{\smash{{\SetFigFont{12}{14.4}{\rmdefault}{\mddefault}{\updefault}{\color[rgb]{0,0,0}$t$}%
}}}}
\put(2522,-3474){\makebox(0,0)[lb]{\smash{{\SetFigFont{12}{14.4}{\rmdefault}{\mddefault}{\updefault}{\color[rgb]{0,0,0}$t+1$}%
}}}}
\put(1933,-4860){\makebox(0,0)[lb]{\smash{{\SetFigFont{12}{14.4}{\rmdefault}{\mddefault}{\updefault}{\color[rgb]{0,0,0}$\pi_L\rho$}%
}}}}
\end{picture}%

%% file: loaded.pdf_t
\begin{picture}(0,0)%
\includegraphics{loaded.pdf}%
\end{picture}%
\setlength{\unitlength}{4144sp}%
\begingroup\makeatletter\ifx\SetFigFont\undefined%
\gdef\SetFigFont#1#2#3#4#5{%
  \reset@font\fontsize{#1}{#2pt}%
  \fontfamily{#3}\fontseries{#4}\fontshape{#5}%
  \selectfont}%
\fi\endgroup%
\begin{picture}(2057,2049)(664,-4573)
\put(2379,-2937){\makebox(0,0)[lb]{\smash{{\SetFigFont{12}{14.4}{\rmdefault}{\mddefault}{\updefault}{\color[rgb]{0,0,0}$\rho$}%
}}}}
\put(1026,-3730){\makebox(0,0)[lb]{\smash{{\SetFigFont{12}{14.4}{\rmdefault}{\mddefault}{\updefault}{\color[rgb]{0,0,0}$\pi_L\rho$}%
}}}}
\end{picture}%

%% file: cases.pdf_t
\begin{picture}(0,0)%
\includegraphics{cases.pdf}%
\end{picture}%
\setlength{\unitlength}{4144sp}%
\begingroup\makeatletter\ifx\SetFigFont\undefined%
\gdef\SetFigFont#1#2#3#4#5{%
  \reset@font\fontsize{#1}{#2pt}%
  \fontfamily{#3}\fontseries{#4}\fontshape{#5}%
  \selectfont}%
\fi\endgroup%
\begin{picture}(7254,2304)(-716,-4828)
\put(4906,-4516){\makebox(0,0)[lb]{\smash{{\SetFigFont{12}{14.4}{\rmdefault}{\mddefault}{\updefault}{\color[rgb]{0,0,0}$\ell$}%
}}}}
\put(5671,-4516){\makebox(0,0)[lb]{\smash{{\SetFigFont{12}{14.4}{\rmdefault}{\mddefault}{\updefault}{\color[rgb]{0,0,0}$\ell$}%
}}}}
\end{picture}%

%% file: regions.pdf_t
\begin{picture}(0,0)%
\includegraphics{regions.pdf}%
\end{picture}%
\setlength{\unitlength}{4144sp}%
\begingroup\makeatletter\ifx\SetFigFont\undefined%
\gdef\SetFigFont#1#2#3#4#5{%
  \reset@font\fontsize{#1}{#2pt}%
  \fontfamily{#3}\fontseries{#4}\fontshape{#5}%
  \selectfont}%
\fi\endgroup%
\begin{picture}(1885,4254)(3364,-7723)
\put(4808,-5647){\makebox(0,0)[lb]{\smash{{\SetFigFont{12}{14.4}{\rmdefault}{\mddefault}{\updefault}{\color[rgb]{0,0,0}$C$}%
}}}}
\end{picture}%